\documentclass[12pt]{article}
\usepackage{latexsym,amsmath,amsopn,amssymb,amsthm,amsfonts}

\newtheorem{theorem}{Theorem}[section]
\newtheorem{corollary}[theorem]{Corollary}

\newtheorem{lemma}[theorem]{Lemma}

\newcommand\g{{\mathfrak g}}

\newcommand{\R}{\mathbb{R}}

\begin{document}

{\bf \large
\centerline{N.~K.~Smolentsev}

\vspace{3mm}
\centerline{Canonical pseudo-K\"{a}hler structures }
\centerline{on six-dimensional nilpotent Lie groups\footnote{The work was partially supported by RFBR 12-01-00873-a and by Russian President Grant supporting scientific schools SS-544.2012.1
}}}

\vspace{3mm}

\begin{abstract}
In this paper we consider left-invariant pseudo-K\"{a}hler structures on six-dimensional nilpotent Lie algebras. The explicit expressions of the canonical complex structures are calculated, and the curvature properties of the associated pseudo-K\"{a}hler metrics are investigated.
It is proved that the associated pseudo-K\"{a}hler metric is Ricci-flat, that the curvature tensor has zero pseudo-Riemannian norm, and that the curvature tensor has some non-zero components that depend only on two or, at most, three parameters. The pseudo-K\"{a}hler structures obtained give basic models of pseudo-K\"{a}hler six-dimensional nilmanifolds.
\end{abstract}

\section{Preface} \label{Preface}

A left-invariant pseudo-K\"{a}hler (or indefinite K\"{a}hler) structure $(J, \omega)$ on a Lie group $G$ with Lie algebra $\g$ consists of a nondegenerate closed left-invariant 2-form $\omega$ and a left-invariant complex structure $J$ on $G$ which are \emph{compatible}, i.e. $\omega(JX, JY) = \omega(X, Y)$, for all $X, Y \in \g$.

Given a pseudo-K\"{a}hler structure $(J, \omega)$ on $\g$, there exists an associated nondegenerate symmetric 2-tensor $g$ on $\g$ defined by $g(X, Y) = \omega(X,JY)$ for $X, Y \in \g$. It is well-known \cite{BG} that if the Lie algebra $\g$ is nilpotent, then the associated metric $g$ for any compatible pair $(J, \omega)$ cannot be positive definite unless $\g$ is abelian. Therefore $g$ is a pseudo-Riemannian metric. We shall say that $g$ is a left-invariant pseudo-K\"{a}hler metric on the Lie group $G$.

A classification of left-invariant complex structures on six-dimensional nilpotent Lie groups is given in the work of Salamon \cite{Sal-1}.
The classification of left-invariant symplectic structures on six-dimensional nilpotent Lie groups is established in the paper by Goze, Khakimdjanov and Medina \cite{Goze-Khakim-Med}. In the paper by Cordero, Fern\'{a}ndez and Ugarte \cite{CFU2}, the pseudo-K\"{a}hler structures in a six-dimensional nilpotent Lie group are studied.
It is shown that for a pseudo-K\"{a}hler structure $(J, \omega)$ on a six-dimensional nilpotent Lie group, the complex structure $J$ will be nilpotent, and on some Lie groups abelian.

Let $\{e^1,\dots,e^n\}$ be a basis of left-invariant 1-forms on $G$. In the following theorem we write $\mathfrak{g}$ as an $n$-tuple $(0, 0, de^3, . . . , de^n)$, abbreviating $e^{ij}=e^ {i} \wedge e^ {j}$  further to $ij$. For example, the $n$-tuple $(0,0,0,0,12,34)$ designates a Lie algebra with the equations: $de^1=0$, $de^4=0$, $de^5 =e^1\wedge e^2$ and $de^6 =e^3\wedge e^4$. Besides, for each algebra its number in the classification list for symplectic Lie algebras given in \cite{Goze-Khakim-Med} is also specified.

\begin{theorem} [\cite {CFU2}] \label {PsKahl_6}
Let $\mathfrak {g}$ be a (nonabelian) six-dimensional nilpotent Lie algebra. Then $\mathfrak {g}$ possesses a compatible pair $(J, \omega)$ if and only if $\mathfrak {g}$ is isomorphic to one of the following Lie algebras:
$$
\begin{array}{ll}
  \g_{21} = (0, 0, 0, 0, 12, 14 + 25), & \g_{24} = (0, 0, 0, 0, 12, 34), \\
  \g_{14} = (0, 0, 0, 12, 13, 14), & \g_{17} = (0, 0, 0, 0, 12, 14 + 23), \\
  \g_{13} = (0, 0, 0, 12, 13, 14 + 23), & \g_{16}=(0, 0, 0, 0, 13 + 42, 14 + 23),\\
  \g_{15} = (0, 0, 0, 12, 13, 24), & \g_{23} = (0, 0, 0, 0, 12, 13), \\
  \g_{11} = (0, 0, 0, 12, 13 + 14, 24), & \g_{18} = (0, 0, 0, 12, 13, 23), \\
  \g_{10} = (0, 0, 0, 12, 14, 13 + 42), & \g_{25} = (0, 0, 0, 0, 0, 12), \\
  \g_{12} = (0, 0, 0, 12, 13 + 42, 14 + 23). &  \\
\end{array}
$$
\end{theorem}

\begin{corollary} [\cite {CFU2}]
In dimension 6, the Lie algebra $\g$ has compatible pairs $(J,\omega)$ if and only if
it admits both symplectic and nilpotent complex structures.
\end{corollary}

For all the Lie algebras listed, the compatible complex structure $J$ is nilpotent, and for algebras $\g_{21}$, $\g_{12}$, $\g_{16}$ and $\g_{25}$ $J$ is abelian. For each Lie algebra in this list an example of a nilpotent complex structure is given in \cite{CFU2}, and the compatible symplectic forms for it are presented.

It is more natural to start with the classification list of Goze, Khakimdjanov and Medina \cite{Goze-Khakim-Med} in which all symplectic 6-dimensional Lie algebras are presented and where it is shown that each nilpotent Lie algebra is symplectoisomorphic to one of the algebras on this list.

Therefore we will consider the Lie algebras from theorem \ref{PsKahl_6} with the symplectic structure from the list in \cite{Goze-Khakim-Med} and from these we will search for all compatible complex structures.
Generally speaking, for a given symplectic structure $\omega$ there is a multiparametrical set of complex structures. For example, for the abelian group $\R^6$ with symplectic structure $\omega = e^1\wedge e^2 +e^3\wedge e^4 +e^5\wedge e^6$, there is a 12-parametrical set compatible with $\omega$ complex structures $J=(\psi_{ij})$.
However the curvature of the associated metric $g(X,Y)=\omega(X,JY)$ will be zero for any complex structure $J$. Therefore, from the geometrical point of view, there is no sense in considering the general compatible complex structure $J$. It is much more natural to choose the elementary one: $J(e_1)=e_2$,\ $J(e_3)=e_4$,\ $J(e_5)=e_6$. We will adhere to this point of view for all Lie algebras g.

After finding a multiparametrical set of complex structures $J=(\psi_{ij})$, compatible with the symplectic structure $\omega$ on $\g$, we will consider all free parameters $\psi_{ij}$ on which the curvature of the associated metric does not depend to be zero. We will call such structures \emph{canonical}.

We will obtain explicit expressions for the canonical complex structures, and we will investigate the curvature properties of the associated pseudo-K\"{a}hler metrics. We will prove that the associated pseudo-K\"{a}hler metric is Ricci-flat, that the scalar square $g(R,R)$ of the curvature tensor $R$ is equal to zero and that the curvature tensor has some non-zero components that depend only on two or, at the most, three parameters.

For some Lie groups in this paper, to find the compatible complex structure $J=(\psi_{ij})$ we will use the Magnin complex structures discovered in \cite{Mag-3}, and we will solve only the compatible condition $\omega(JX,JY)=\omega(X,Y)$. For the remaining Lie groups the compatible complex structure $J=(\psi_{ij})$ is a solution of three equations: the compatible condition, $\omega \circ J+J^t\circ \omega=0$; the condition of an almost complex structure, $J^2=-Id$; the integrability $N_J(X,Y)=0$.
First we solve a linear set of equations under the compatible condition. If there are difficulties in finding a solution to the equations for the integrability, we calculate a tensor of curvature of the associated metric of the almost complex structure, and we find the parameters on which it depends. Setting the remaining free parameters equal to zero, we solve the equations for the integrability, and we discover the required complex structure and the associated metric.

All computations are fulfilled in system of computer mathematics Maple.
Formulas for computations are specified in the end of this paper.

\section{Pseudo-K\"{a}hler structures on nilpotent Lie groups}
Let $G$ be a real Lie group of dimension $n$ with Lie algebra $\mathfrak{g}$.
In this paper we will study a left-invariant pseudo-K\"{a}hler structure $(J, \omega)$ on a Lie group $G$.
A left-invariant pseudo-K\"{a}hler (or indefinite K\"{a}hler) structure $(J, \omega)$ on a Lie group $G$ with Lie algebra $\g$ consists of a nondegenerate closed left-invariant 2-form $\omega$ and a left-invariant complex structure $J$ on $G$ which are \emph{compatible}, i.e. $\omega(JX, JY) = \omega(X, Y)$, for all $X, Y \in \g$. Given a pseudo-K\"{a}hler structure $(J, \omega)$ on $\g$, there exists an associated left-invariant pseudo-K\"{a}hler metric $g$ on $\g$ defined by
$$
g(X, Y) = \omega(X,JY), \mbox{ for } X, Y \in \g.
$$

As the symplectic structure $(J, \omega)$ and the complex structure $J$ are left-invariant on $G$, they are defined by their values on the Lie algebra $\g$. Therefore from now on we will deal only with the Lie algebra $\g$ and we will define $\omega$ and $J$ as the symplectic and, respectively, the complex structures on the Lie algebra $\g$.

An almost complex structure on a Lie algebra $\g$ is an endomorphism $J:\g \rightarrow \g$ satisfying $J^2 = -I$, where $I$ is the identity map. The integrability condition of a left-invariant almost complex structure $J$ on $G$ is expressed in terms of the Nijenhuis tensor $N_J$ on $\g$:
\begin{equation}\label{Nij1}
N_J(X,Y) = [JX, JY] - [X,Y] - J[JX,Y] - J[X, JY], \mbox{ for all } X, Y \in \g.
\end{equation}
An almost complex structure $J$ on $\g$ is called \emph{integrable} if $N_J \equiv 0$. In this case $J$ is called a \emph{complex structure} on $G$.

The descending central series $\{C^k({g})\}$ of $\mathfrak{g}$ is defined inductively by
$$
C^0({\g})={g},\quad C^k({\g})=[\,{\g},C^{k-1}({\g})],\ k>0.
$$
The Lie algebra $\mathfrak{g}$ is said to be \emph{nilpotent} if $C^k(\g) =0$ for some $k$.
The Lie group is said to be nilpotent if its Lie algebra is nilpotent.

If $\g$ is $s$-step nilpotent, then the ascending central series $\mathfrak{g}_0=\{0\}\subset \mathfrak{g}_1 \subset \mathfrak{g}_2\subset\dots \subset \mathfrak{g}_{s-1}\subset \mathfrak{g}_s =\mathfrak{g}$ of $\g$ is defined inductively by
\[
\mathfrak{g}_k=\{X\in \mathfrak{g} |\ [X,\mathfrak{g}]\subseteq  \mathfrak{g}_{k-1}\}, \ k\geq 1.
\]
It is well-known that the sequence $\{\g_k\}$ increases strictly until $\g_s$ and, in particular, that the ideal $\g_1$ is the center $\mathcal{Z}$ of Lie algebra $\g$.

Since the spaces $\{\g_k\}$ are not, in general, $J$-invariant, the sequence is not suitable for working with $J$. We introduce a new sequence $\{\mathfrak{a}_l(J)\}$ having the property of $J$-invariance \cite{CFGU01}. The ascending series $\{\mathfrak{a}_l(J)\}$ of the Lie algebra $\g$, compatible with the left-invariant complex structure $J$ on $G$, is defined inductively as follows: $\frak{a}_0(J) = \{0\}$,
$$
\mathfrak{a}_l(J) =\{X\in \mathfrak{g}\ |\ [X,\mathfrak{g}]\subseteq  \mathfrak{a}_{l-1}(J) \mbox { and } [JX,\mathfrak{g}]\subseteq  \mathfrak{a}_{l-1}(J) \},\ l\ge 1.
$$
It is clear that each $\mathfrak{a}_l(J)$ is a $J$-invariant ideal in $\g$ and $\mathfrak{a}_k(J) \subseteq \mathfrak{g}_{k}$ for $k\ge 1$. It must be noted that the terms $\mathfrak{a}_l(J)$ depend on the complex structure $J$ considered on $G$. Moreover, this ascending series, in spite of $\g$ being nilpotent, can stop without reaching the Lie algebra $\g$, that is, it may happen that $\mathfrak{a}_l(J) \neq \g$ for all $l$. The following definition is motivated by this fact.

The left-invariant complex structure $J$ on $G$ is called \emph{nilpotent} if there is a number $p$ such that $\mathfrak{a}_p(J)=\mathfrak{g}$.

It is obvious that the ideal $\mathfrak{a}_1(J)$ lies in the center $\mathcal{Z}$ of the Lie algebra $\g$. If the nilpotent Lie algebra has the two-dimensional center $\mathcal{Z}$ for any left-invariant complex nilpotent structure $J$, the ideal $\mathcal{Z}$ is $J$-invariant. If the nilpotent Lie algebra has an increasing central sequence of ideals $\mathfrak{g}_k$, $k=0,1,\dots, s$, for which the dimension increases each time by two units for any left-invariant complex nilpotent structure $J$, then the equalities $\mathfrak{g}_k =\frak{a}_k(J)$, $k=0,1,\dots, s$ are fulfilled.
If the basis of $\g$ is chosen so that $\mathfrak{g}_1 =\{e_{2n-1},\, e_{2n}\}$, $\mathfrak{g}_2 =\{e_{2n-3},\, e_{2n-2},\, e_{2n-1},\, e_{2n}\}$, \dots, then the complex structure $J$ has the following block type (for example, for a six-dimensional case):
\begin{equation} \label{J6Bl}
J_0=  \left( \begin {array}{cccccc} \psi_{11}&\psi_{12}&0&0&0&0 \\
\psi_{21}&\psi_{22}&0&0&0&0\\
\psi_{31}&\psi_{32}&\psi_{33}&\psi_{34} &0&0\\
\psi_{41}&\psi_{42}&\psi_{43}&\psi_{44} &0&0\\
\psi_{51}&\psi_{52}&\psi_{53}&\psi_{54} &\psi_{55}&\psi_{56}\\
\psi_{61}&\psi_{62}&\psi_{63}&\psi_{64} &\psi_{65}&\psi_{66}
\end {array} \right).
\end{equation}

The remaining parameters in (\ref{J6Bl}) are not free, because they are restricted by the integrability conditions $N_J=0$ and $J^2=-1$.

Let us establish some properties of a symplectic nilpotent Lie algebra $(\g,\omega)$ with a compatible almost complex structure $J$.

\begin{lemma} \label {C1ortZ}
Let $\mathcal{Z}$  be the Lie algebra center, then for any symplectic form $\omega$ on $\mathfrak{g}$, the equality $\omega (C^1\mathfrak {g}, \mathcal {Z}) =0$ is fulfilled
\end{lemma}

This follows at once from the formula
$d\omega(X,Y,Z)=\omega([X,Y],Z) -\omega([X,Z],Y) +\omega([Y,Z],X)=0$, \ $\forall X,Y\in\mathfrak{g}$, $\forall Z\in\mathcal{Z}$.

\begin {corollary}
If the ideal $C^k\g $, \, $k\ge 1$, is  $J$-invariant, then any vector $X\in\mathcal {Z} \cap C^k\g$ is isotropic for the associated metric $g(X, Y) = \omega(X,JY)$. In particular, the associated metric is pseudo-Riemannian.
\end {corollary}

\begin{corollary}
$\omega (C^1\mathfrak{g} \oplus J (C^1\mathfrak {g}), \mathfrak{a}_1(J)) =0$.
\end{corollary}

\begin {corollary} \label {1.6}
For any (pseudo) K\"{a}hler structure $(\mathfrak{g}, \omega, g, J) $, the ideal $\mathfrak{a}_1(J) \subset \mathcal{Z} $ is orthogonal to a subspace $C^1\mathfrak{g}\oplus J(C^1\mathfrak {g})$:
$$
g(C^1\mathfrak{g}\oplus J(C^1\mathfrak{g}),\mathfrak{a}_1(J)) =0.
$$
\end {corollary}

From the formula
\[
2g (\nabla_X Y, Z) {=} g ([X, Y], Z) {+} g ([Z, X], Y) {+} g (X, [Z, Y])
\]
for a covariant derivative $\nabla$ of left-invariant vector fields, the following observations are at once implied:
\begin{itemize}
\item if the vectors $X $ and $Y $ lie in the center $\mathcal{Z}$ of the Lie algebra $\g$, then $ \nabla_X Y=0$ for any left-invariant (pseudo) Riemannian metric $g$ on the Lie algebra; \
\item if the vector $X$ lie in the center $\mathcal{Z}$ of the Lie algebra $\g$, then $\nabla_X Y = \nabla_Y X$.
\end{itemize}

\begin{lemma} \label {1.7}
If the vector $X$ lies in an ideal $\mathfrak{a}_1(J)\subset \mathcal{Z}$, then $\nabla_X Y=\nabla_Y X=0$,\ $\forall Y\in \mathfrak{g}$.
\end{lemma}

\begin{proof}
Let $X\in \mathfrak{a}_1(J)\subset \mathcal{Z}$ and $Z,Y\in \mathfrak{g}$. Then corollary \ref {1.6} implies that $2g(\nabla_X Y,Z) = g(X,[Z,Y])=0$.
\end{proof}

Let $R(X,Y)Z = \nabla_X  \nabla_Y Z - \nabla_Y  \nabla_X Z -\nabla_{[X,Y]} Z$
be the curvature tensor.

\begin{corollary} \label{1.8}
If $X\in \mathfrak{a}_1(J) \subset \mathcal{Z}$, then $R(X,Y)Z = R(Z,Y) X=0$,  for all  $Y, Z\in \mathfrak {g} $.
\end{corollary}

\begin{corollary} \label{ParRiem}
Let $J$ be a compatible almost complex structure on the symplectic nilpotent Lie algebra $(\g, \omega)$, and let $g=\omega \cdot J$ be the associated (pseudo) Riemannian metric. Let us choose a basis $\{e_1, e_2,\dots e_{2p}, e_{2p+1},\dots, e_{2m-1}, e_{2m}\}$ of the Lie algebra $\g$, such that $\{e_{2p},e_{2p+1}\dots e_{2m-1},e_{2m}\}$ is a basis of the ideal $\mathfrak{a}_1(J) \subset \mathcal{Z}$.
Let $J = (\psi_{ij})$ be the matrix of $J$ in this basis. Then for all $X,Y\in \g$, the covariant derivative $\nabla_X Y$ does not depend on the free parameters $\psi_{ij}$, where $i=2p,\dots 2m$, and for all $j$. In particular, the curvature tensor does not depend on these parameters $\{\psi_{ij}\}$, $i=2p,\dots 2m,\, j=1,\dots, 2m$.
\end{corollary}

\begin{proof}
For all $X\in \g$, let $JX = J_1X +J_aX$, where $J_1X \in \mathbb{R}\{e_1, e_2,\dots, e_{2p-1}\}$ and $J_aX\in \mathbb{R}\{e_{2p}, e_{2p+1},\dots, e_{2m-1}, e_{2m}\} = \mathfrak{a}_1(J)\subset \mathcal{Z}$.
From the covariant derivative formula it follows that for all $Z\in \g$:
$$
2g(\nabla_X Y, Z)=2\omega(\nabla_X Y, JZ)= \omega([X,Y],JZ) + \omega([Z,X],JY)  +\omega([Z,Y],JX),
$$
$$
2\omega(\nabla_X Y, Z)= \omega([X,Y],Z) - \omega([JZ,X],JY)  -\omega([JZ,Y],JX)=
$$
$$
= \omega([X,Y],Z) + \omega([J_1Z+J_aZ,X],J_1Y+J_aY)  +\omega([J_1Z+J_aZ,Y],J_1X+J_aX)=
$$
$$
=\omega([X,Y],Z) + \omega([J_1Z,X],J_1Y)  +\omega([J_1,Y],J_1X).
$$
As the component $J_1$ does not depend on the parameters $\{\psi_{ij}\}$, $i=2p,\dots 2m,\, j=1,\dots, 2m$, it is also true that the covariant derivative $\nabla_X Y$ does not depend on them.
\end{proof}

\begin{corollary} \label{ParRiem6}
For the six-dimensional case, if the nilpotent complex structure $J$ compatible with $\omega$ is of the form in (\ref{J6Bl}), then the curvature tensor $R(X,Y)Z=\nabla_X \nabla_Y Z -\nabla_Y \nabla_X Z -\nabla_{[X,Y]}Z$ of the associated metric does not depend on the free parameters $\psi_{5j}$, $\psi_{6j}$, $j=1,\dots, 6$.
\end{corollary}

\textbf{Remark 1.} The parameters $\psi_{ij}$ of the complex structure $J$ are restricted by three conditions: a compatibility condition, an integrability condition and the fact that $J^2 =-1$. Therefore, from what is specified above, some of the parameters can be express through others. In corollaries \ref{ParRiem} and \ref{ParRiem6} it this is a question of free parameters, i.e., parameters which remain independent. Curvature does not depend on them.
 \vspace{1mm}

For a six-dimensional nilpotent Lie algebra $\g$ which possesses a pseudo-K\"{a}hler structure, the dimensions of its increasing central sequence $\mathfrak{g}_k$ can be: (2,4,6), (2,6), (3,6), (4,6) and 6.
We will  call the sequence of these dimensions the \emph{Lie algebra type}. In the list of Lie algebras in theorem \ref{PsKahl_6}, the Lie algebras with type (2,4,6) are at the beginning, and are the first seven Lie algebras.

Let us consider the nilpotent Lie algebras for which the sequence of ideals $\mathfrak{g}_1\subset \mathfrak{g}_2\subset\mathfrak{g}_3 =\mathfrak{g}$ has the dimensions (2,4,6). It is easy to see that such a Lie algebra of type (2,4,6) is decomposed into the direct sum of two-dimensional subspaces:
$$
\mathfrak {g} = A \oplus B \oplus \mathcal {Z},
$$
with properties:
\begin{itemize}
\item $\mathcal{Z} = \mathfrak{g}_1$, the Lie algebra center,
\item $B \oplus \mathcal{Z} = \mathfrak{g}_2$,
\item $[A, A] \subset B \oplus \mathcal{Z}$, \ $ [A, B] \subset \mathcal{Z}$.
\end{itemize}

Let us consider further that in $\mathfrak{g}$ the basis $e_1,\dots, e_6$  is chosen so that   $\{e_1,e_2\}$, $\{e_3,e_4\}$ and $\{e_5,e_6\}$ are the bases of the subspaces $A$, $B$ and $\mathcal{Z}$ respectively.

For any nilpotent complex structure $J $ on algebra of type $ (2,4,6), $ the sequence of ideals $\mathfrak {a}_k(J) $ coincides with $\mathfrak {g}_k $, \ $k=1,2,3$ and the matrix $J$ has a block shape (\ref {J6Bl}).
For this complex structure $J$ we also have $C^1\mathfrak{g}\oplus J(C^1\mathfrak{g})=B\oplus \mathcal{Z}=\mathfrak{g}_2$.

The subspace $W\subset \mathfrak {g}$ is called $\omega$-\emph{isotropic} if and only if $\omega(W, W) =0$.
We will call subspaces $U, V \subset \mathfrak{g}$ $\omega$-\emph{dual} if, for any vector $X\in U$ there is a vector $Y\in V$ such that $\omega (X, Y) \ne 0$ and, on the contrary, $\forall Y\in V$, $\exists X\in U$, such that $\omega (X, Y) \ne 0$.

\begin{theorem} \label{246}
Let the six-dimensional symplectic Lie algebra $(\mathfrak{g}, \omega)$ have type (2,4,6) and
\begin{equation} \label{abz}
\mathfrak {g} = A \oplus B \oplus \mathcal {Z},
\end{equation}
where $B \oplus \mathcal{Z} = \mathfrak{g}_2$ is an abelian subalgebra.
We will assume that the subspaces $A$ and $\mathcal{Z}$ are $\omega$-isotropic and $\omega$-dual, and that on the subspace $B$ the form $\omega$ is nondegenerate.
Then for any nilpotent complex structure $J$ compatible with $\omega$, and the Levi-Civita connection $\nabla$ of the associated pseudo-Riemannian metric $g=\omega\cdot J$, the following properties are fulfilled:
\begin{itemize}
\item $\nabla_X Y\in B \oplus \mathcal{Z},\quad \forall\,  X,Y\in A$,
\item $\nabla_X Y, \nabla_Y X\in \mathcal{Z}$,\quad for all \ $X\in A,\, Y\in B$,

\item $\nabla_X Y = \nabla_Y X =0$,\quad for all \ $X\in A,\, Y\in \mathcal{Z}$,

\item  $\nabla_X Y=0,\quad \forall \, X,Y\in B\oplus \mathcal{Z}$.
\end{itemize}
\end{theorem}

\begin{proof}
Let $X, Y\in A$. If $\nabla_X Y$ has a nonzero component from $A$ then there is a vector $JZ\in \mathcal {Z} $, such that $ \omega (\nabla_X Y, JZ) \ne 0$. On the other hand, $2\omega (\nabla_X Y, JZ) = 2g(\nabla_X Y, Z) = g([X, Y], Z) + g([Z,X], Y) + g([Z, Y], X) = \omega ([X, Y], JZ) =0$, as, from lemma \ref{C1ortZ}, $ \omega (C^1\mathfrak {g}, \mathfrak {a} _1 (J)) =0$.

Let now $X\in A$ and $Y\in B$.
It can be shown in an exactly similar way that $\nabla_X Y$ has a zero component from $A$.
We will assume that $\nabla_X Y $ has a non-zero component from $B$.
Then there is a vector $Z\in B\oplus \mathcal{Z}$ such that $JZ\in B$ and such that  $\omega (\nabla_X Y, JZ) \ne 0$.
At the same time, $2\omega(\nabla_X Y,JZ) = 2g(\nabla_X Y,Z) = g([X,Y],Z) +g([Z,X],Y)  = \omega([X,Y],JZ)+ \omega([Z,X],JY) =0$.
The last equality follows from the commutativity of $B\oplus \mathcal{Z}$, and from this that    $Y,JY,Z,JZ\in B \subset  C^1\mathfrak{g} \oplus J(C^1\mathfrak{g})$, so then $[X,Y],[Z,X] \in \mathcal{Z}$ and  $\omega(C^1\mathfrak{g} \oplus  J(C^1\mathfrak{g}),\mathcal{Z})=0$. $\nabla_Y X =\nabla_X Y - [X, Y]\in \mathcal{Z}$.

Let us consider the third statement. Let $X\in A$ and $Y\in \mathcal{Z}$. Then for any $Z\in \mathfrak{g}$,\ $2g(\nabla_X Y,Z) = g([X,Y],Z) +g([Z,X],Y) +g([Z,Y],X) = \omega([Z,X],JY) =0$ by the same arguments as for the previous point.

Let us consider the last statement. Let $X,Y\in B \oplus \mathcal{Z}$. Then for any $Z\in \mathfrak{g}$, $2g(\nabla_X Y,Z) = g([X,Y],Z) +g([Z,X],Y) +g([Z,Y],X) = \omega([Z,X],JY)+ \omega([Z,Y],JX) =0$ by the same arguments as for the previous point.
\end{proof}

\begin{corollary} \label{3.3}
Under the suppositions of theorem \ref{246}, if the vector $X$ lies in an ideal $\mathfrak{a}_2(J)=B\oplus\mathcal{Z}$, then $R(X, Y)Z = R(Z,Y)X =0$, \  for all $Y, Z\in \mathfrak{g}$.
\end{corollary}

\begin{corollary} \label{Ricci-zero}
Under the suppositions of theorem \ref{246}, for any $X, Y, Z\in \mathfrak{g}$, the condition $R(X, Y)Z\in \mathcal{Z}$ is fulfilled.
Therefore the pseudo-Riemannian norm of the curvature tensor is equal to zero. According to the expansion $\mathfrak{g}=A\oplus B\oplus \mathcal{Z}$ we choose the bases $\{e_1, e_2\}$, $\{e_3, e_4 \}$ and $\{e_5, e_6 \}$.
Then the curvature tensor can have within its symmetries only four non-zero components, $R_{1,2,1}^5$, $R_{1,2,1}^6$, $R_{1,2,2}^5$, $R_{1,2,2}^6$. In particular, the Ricci tensor is equal to zero.
\end{corollary}

\begin{corollary} \label{Par3.3}
Under the suppositions of theorem \ref{246}, if the compatible almost complex structure $J$ has a block type (\ref {J6Bl}), then the curvature tensor of the associated metric $g=\omega \cdot J$ does not depend on the free parameters $\psi_{i1}$ and $\psi_{i2}$,for $i=3,4,5,6$.
\end{corollary}

Let us notice that parameters $\psi_{ij}$ of the complex structure $J$ are connected by three conditions: a compatibility condition, an integrability condition and the fact that $J^2=-1$. Therefore from what has been specified above some of the parameters can be expressed through others.
If, as a result, among $\psi_{i1}$ and $\psi_{i2}$, $i=3,4,5,6$, there were independent parameters, it would be possible to set them at zero as the curvature does not depend on them. We remember that, according to corollary \ref{ParRiem6}, the curvature $R(X,Y)$ of the associated metric also does not depend on the free parameters $\psi_{51}$, $\psi_{52}$, $\psi_{53}$, $\psi_{54}$, $\psi_{61}$,  $\psi_{62}$,  $\psi_{63}$, $\psi_{64}$.

 \vspace {1mm}
Similar statements are true for type (2,6) Lie algebras. A type (4,6) Lie algebra is the direct product of a four-dimensional Lie algebra and $\mathbb{R}^2$. The case (3,6) is the most complicated.

 \vspace {1mm}
\textbf{Remark 2.} There is an obvious generalization of theorem \ref{246} to the case where $\dim \g > \, 6$.
It is necessary to assume that there exist sequences of ideals $\g_1\subset \g_2\subset \dots \subset \g_n = \g$ with dimensions of $2,4, \dots, n $, invariant under $J$.
Choosing additional two-dimensional subspaces $A_i$ to these ideals, we obtain the expansion $\g =A_1\oplus A_2\oplus \dots \oplus A_n =\g_1 =\mathcal{Z}$. The form of $\omega$ should be such that on each subspace $A_i$ (except, maybe, on $A_{n/2}$) it is degenerated, and also the subspaces $A_i$ and $A_{n-i}$ are $\omega$-dual.

 \vspace {1mm}
\textbf{Remark 3.} All calculations are made in the Maple system using the formulas specified at the end of the paper.

\section{Lie algebras of type $(2,4,6)$}
Let us consider all six-dimensional nilpotent Lie algebras of type (2,4,6). According to theorem \ref{PsKahl_6}, there are seven such Lie algebras which possess a pseudo-K\"{a}hler structure. We remember that the number of each algebra corresponds to its number in the classification list in \cite{Goze-Khakim-Med}.

\subsection{The Lie group $G_{14}$}
Let us consider a six-dimensional Lie group $G_{14}$ which has a Lie algebra $\g_{14}$ with non-trivial Lie brackets (see \cite{Goze-Khakim-Med}:  $[X_{1},X_{2}]=X_{4}$,\ $[X_{1},X_{4}]=X_{6}$,\  $[X_{1},X_{3}]=X_{5}$. The algebra $\g_{14}$ has (see \cite{Goze-Khakim-Med}) three symplectic structures which are noted using the dual base $\{\alpha^{i}\}$ as follows:

$\omega _{1}=\alpha^{1}\wedge \alpha^{6}+\alpha^{2}\wedge \alpha^{4}+\alpha^{3}\wedge \alpha^{5}$,

$\omega _{2}=\alpha^{1}\wedge \alpha^{6}-\alpha^{2}\wedge \alpha^{4} +\alpha^{3}\wedge \alpha^{5}$,

$\omega _{3}=\alpha^{1}\wedge \alpha^{6}+\alpha^{2}\wedge \alpha^{5}+\alpha^{3}\wedge \alpha^{4}$.\\
Left-invariant complex structures on this group are found in an explicit form in the work of Magnin \cite{Mag-3} (the algebra $M1$). We will use Magnin's results, so we will make the replacement: $X_2=-e_1$,\ $X_1=e_2$,\ $X_3=e_3$,\ $X_4=e_4$,\ $X_6=e_5$,\ $X_5=e_6$. Then the non-trivial brackets are given in \cite{Mag-3}:

 $[e_1,e_2]=e_4$, $[e_2,e_3] = e_6$, $[e_2,e_4] = e_5$,\\
and the symplectic forms become:

$\omega _{1}=-e^{1}\wedge e^{4}+e^{2}\wedge e^{5} +e^{3}\wedge e^{6}$,

$\omega _{2}=e^{1}\wedge e^{4} +e^{2}\wedge e^{5}+e^{3}\wedge e^{6}$,

$\omega _{3}=-e^{1}\wedge e^{6}+e^{2}\wedge e^{5}+e^{3}\wedge e^{4}$.

In \cite{Mag-3} it is shown that the Lie group $G_{14}$ has a 10-parametrical set of left-invariant complex structures. A direct check of the compatible property $\omega(JX,Y)+\omega(X,JY)$, $\forall \, X,Y\in \g$, shows that for first two symplectic forms there are no compatible complex structures on the group $G_{14}$. For the form $\omega_3$ the compatible complex structure depends on 6 parameters and is of the form:
$$
J = \left(\begin{array}{cccccc}
      \psi_{11} & \psi_{12} & 0 & 0 & 0 & 0 \\
      -\frac{\psi_{11}^2+1}{\psi_{12}} & -\psi_{11} & 0 & 0 & 0 & 0 \\
      \frac{\psi_{42}(\psi_{11}^2+1) -2\psi_{41}\psi_{12}\psi_{11}}{\psi_{12}^2} & -\psi_{41} & -\psi_{11} & -\frac{\psi_{11}^2+1}{\psi_{12}} & 0 & 0 \\
      \psi_{41} & \psi_{42} & \psi_{12} & \psi_{11} & 0 & 0 \\
      \psi_{51} & J_{52} & \psi_{42} & \psi_{41} & \psi_{11} & \psi_{12}\\
      \psi_{61} & -\psi_{51}  & -\psi_{41} & \frac{\psi_{42}(\psi_{11}^2+1)-
      2\psi_{41}\psi_{12}\psi_{11}}{\psi_{12}^2} & -\frac{\psi_{11}^2+1}{\psi_{12}}& -\psi_{11} \\
    \end{array}\right),
$$
where $J_{52}= \frac{-2\psi_{11}\psi_{12}(\psi_{42}\psi_{41} -\psi_{12}\psi_{51}) +\psi_{42}^2(\psi_{11}^2 +1) +\psi_{12}^2(\psi_{41}^2 +\psi_{12}\psi_{61})}{(\psi_{11}^2+1)\psi_{12}}$.

The curvature tensor of the metric $g(X,Y)=\omega_3(X,JY)$ is equal to zero for all values of the parameters. Therefore we choose the elementary pseudo-K\"{a}hler structure with zero values of the free parameters and $\psi_{12}=-1$. Then the canonical pseudo-K\"{a}hler structure is set as follows:

$J(e_1) = e_2,\quad J(e_3) = e_4,\quad
J(e_5) = e_6$,

$g=-e^1\,e^5 -e^2\,e^6 -(e^3)^2 -(e^4)^2.$


\subsection{The Lie group $G_{21}$}
The Lie algebra $\g_{21}$ is defined by: $[e_1,e_2] = e_4$, $[e_1,e_4] = e_6$, $[e_2,e_3] = e_6$. This Lie algebra has two symplectic structures \cite{Goze-Khakim-Med}:

$\omega_1 = e^1\wedge e^6 +e^2\wedge e^4 -e^3\wedge e^4 -e^3\wedge e^5,$

$\omega_2 = e^1\wedge e^6 +e^2\wedge e^5 -e^3\wedge e^4.$

A direct check shows that for the first structure $\omega_1 = e^1\wedge e^6 +e^2\wedge e^4 -e^3\wedge e^4 -e^3\wedge e^5$ there are no compatible complex structures.
We will consider the second symplectic structure $\omega_2$.
There is a multiparametrical set of compatible complex structures. Taking into account the results of theorem \ref{246} and corollaries \ref{1.8}, \ref{ParRiem6} and \ref{Ricci-zero}, we find by direct evaluation that the tensor of curvature of the associated metric $g(X,Y)=\omega(X,JY)$ depends on two parameters $\psi_{11}$ and $\psi_{12}\ne 0$ and has the following non-zero components:\
$R_{1, 2, 1}^6~=~1~+~\psi_{11}^2$,\
$R_{1, 2, 2}^6 = \psi_{12} \psi_{11}$,\ $R_{1, 2, 1}^5 = \psi_{12}\psi_{11}$,\
$R_{1, 2, 2}^5 = \psi_{12}^2$.
Therefore the semicanonical complex structure is set as follows:

$J(e_1) = \xi_{11}\, e_1 -\frac{\xi_{11}^2+1}{\xi_{12}} \, e_2,\qquad
J(e_2) = \psi_{12}\, e_1 -\psi_{11}\, e_2$,

$J(e_3) = \psi_{11}\, e_3 +\frac{\xi_{11}^2+1}{\xi_{12}}\, e_4,\qquad
J(e_4) = -\xi_{12}\, e_3 -\psi_{11}\, e_4$,

$J(e_5) = \psi_{11}\, e_5 +\frac{\xi_{11}^2+1}{\xi_{12}}\, e_6, \qquad
J(e_6) = -\psi_{12}\, e_5 -\psi_{11}\, e_6$.\\
The corresponding associated metric is:
$$
g {=} \left[\!\!\begin{array}{cccccc}
     0 & 0 & 0 & 0 & \frac{\psi_{11}^2+1}{\psi_{12}} & -\psi_{11} \\
     0 & 0 & 0 & 0 & \psi_{11} & -\psi_{12} \\
     0 & 0 & -\frac{\psi_{11}^2+1}{\psi_{12}} & \psi_{11} & 0 & 0 \\
     0 & 0 & \psi_{11} & -\psi_{12} & 0 & 0 \\
     \frac{\psi_{11}^2+1}{\psi_{12}} & \psi_{11} & 0 & 0 & 0 & 0 \\
     -\psi_{11} & -\psi_{12} & 0 & 0 & 0 & 0 \\
    \end{array}\!\!\right].
$$
On omitting the index of the curvature tensor, it turns out that there is only one (within symmetries) non-zero component of the tensor of curvature $R_{1, 2, 1, 2} =-\psi_{12}$.  Then, supposing $\psi_{12}=-a\ne 0$ and $\psi_{11}=0$, we find the following canonical complex structure and pseudo-K\"{a}hler metric with curvature $R_{1, 2, 1, 2}=a$ on a Lie algebra $\mathfrak{h}_{21}$:

$$
J(e_2) = -a\, e_1,\ J(e_4) = a\, e_3,\ J(e_6) = a\, e_5,
$$
$$
g= -\frac 2a\, e^1\cdot e^5 +2a\, e^2\cdot e^6  -\frac 1a\, (e^3)^2+ a\, (e^4)^2.
$$

\subsection{The Lie group $G_{13}$}
The Lie algebra $\mathfrak{g}_{13}$ is defined by: $[e_1,e_2] =e_4$, $[e_1,e_3]=e_5$, $[e_1,e_4]=e_6$, $[e_2,e_3]=e_6$.
This Lie algebra has three symplectic structures \cite{Goze-Khakim-Med}. Left-invariant complex structures on this group are discovered in an explicit form in the work of Magnin \cite{Mag-3} (algebra $M6$). In order to use the outcomes of Magnin's work, we will rename the base vectors $e_3:=-e_3$, $e_5:=-e_5$, and also find the Lie brackets from Magnin's paper \cite{Mag-3}:

$[e_1,e_2] =e_4$, $[e_1,e_3]=e_5$, $[e_1,e_4]=e_6$, $[e_2,e_3]=-e_6$.

The symplectic structures are:

$\omega_1 =  e^1\wedge e^6 -\lambda e^2\wedge e^5 - (\lambda-1)e^3\wedge e^4$,

$\omega_2 =e^1\wedge e^6 +\lambda e^2\wedge e^4 -e^2\wedge e^5 +e^3\wedge e^5$,

$\omega_3 = e^1\wedge e^6 +e^2\wedge e^4 -\frac 12 e^2\wedge e^5 +\frac 12 e^3\wedge e^4$.

\vspace{1mm}
\textbf{First case.}
We will consider the form $\omega_1= e^1\wedge e^6 -\lambda e^2\wedge e^5 -(\lambda-1)e^3\wedge e^4$. There is a multiparametrical set of the compatible complex structures. Taking into account the outcomes of theorem \ref{246} and corollaries \ref{ParRiem6} and \ref{Ricci-zero}, we find by direct evaluation that the tensor of curvature of the associated metric $g_1(X,Y)=\omega_1(X,J_1Y)$ depends on two parameters $\psi_{11}$ and $\psi_{12}\ne0$ and also has following non-zero components:\
$R_{1, 2, 2}^6 =\frac {(3\lambda -1) \psi_{12} \psi_{11}}{\lambda -1}$,\
$R_{1, 2, 2}^5 =-\frac{(1 + \lambda)(3\lambda -1)\psi_{12}^2}{\lambda (\lambda - 1)}$,\
$R_{1, 2, 1}^6 =\frac{(3 \lambda -1)(1 + \psi_{11}^2)}{\lambda^2 -1}$,\
$R_{1, 2, 1}^5 =-\frac{(3 \lambda - 1) \psi_{12} \psi_{11}}{\lambda (\lambda - 1)}$.
After omitting the index, one non-zero component turns out to be $R_{1, 2, 1, 2} = -\frac{(3 \lambda -1) \psi_{12}}{\lambda -1}$. Therefore the semicanonical complex structure is set out as follows:

$J_1(e_1) = \psi_{11}\,e_1 -\frac{1+{\psi_{11}}^2}{(1+\lambda)\psi_{12}}\, e_2,\quad
J_1(e_2) = (1+\lambda)\psi_{12}\, e_1 -\psi_{11}\, e_2$,\

$J_1(e_3) = \psi_{11}\, e_3 -\frac{1+{\psi_{11}}^2}{\psi_{12}}\, e_4$,\quad
$J_1(e_4) = \psi_{12}\, e_3 -\psi_{11}\, e_4$,\

$J_1(e_5) = \psi_{11}\, e_5 -\frac{\lambda(1+{\psi_{11}}^2)}{(1+\lambda)\psi_{12}}\, e_6,\quad
J_1(e_6) = \frac{(1+\lambda)\psi_{12}}{\lambda}\, e_5 -\psi_{11}\, e_6$.

\vspace{1mm}
The corresponding pseudo-K\"{a}hler metric comes from the formula $g_1=\omega_1\circ J_1$. Suppose that $\psi_{12}=a\ne 0$ and $\psi_{11}=0$. We find the following canonical complex structure and the pseudo-K\"{a}hler metric of curvature $R_{1, 2, 1, 2}=-\frac{(3 \lambda -1) a}{\lambda -1}$ on a Lie algebra $\mathfrak{h}_{13}$:

$$
J_1(e_2) = (1+\lambda)a\, e_1,\quad J_1(e_4) = a\, e_3,\quad
J_1(e_6) = \frac{(1+\lambda)a}{\lambda}\, e_5,
$$
$$
g_1=  \left[ \begin {array}{cccccc} 0&0&0&0&-{\frac {\lambda}{ \left( 1+
\lambda \right) a}}&0\\
\noalign{\medskip}0&0&0&0&0&- \left( 1+\lambda \right) a\\ \noalign{\medskip}0&0&{\frac {\lambda-1}{a}}&0&0&0\\
\noalign{\medskip}0&0&0& \left(\lambda-1 \right) a&0&0\\ \noalign{\medskip}-{\frac {\lambda}{(1+\lambda )a}}&0&0&0&0&0\\ \noalign{\medskip}0&- \left( 1+\lambda \right)a&0&0&0&0
\end {array} \right].
$$

\vspace{1mm}
\textbf{Second case.}
For the symplectic form $\omega_2 =e^1\wedge e^6 +\lambda e^2\wedge e^4 -e^2\wedge e^5 +e^3\wedge e^5$ there are no compatible complex structures.

\vspace{1mm}
\textbf{Third case.}
The symplectic structure is:
$\omega_3 = e^1\wedge e^6 +e^2\wedge e^4 -\frac 12 e^2\wedge e^5 +\frac 12 e^3\wedge e^4$.
For any compatible complex structure and its associated metric, the curvature tensor depends on two parameters $\psi_{11}$ and $\psi_{12}\ne 0$:
$R_{1, 2, 2}^5 =\frac {4 \psi_{12}^2}{3}$, \
$R_{1, 2, 1}^5 =\frac{4 \psi_{12} \psi_{11}}{3}$, \
$R_{1, 2, 2}^6 =-\frac{2 \psi_{12} \psi_{11}}{3}$,\
$R_{1, 2, 1}^6 =-\frac{2(1+ \psi_{11}^2)}{3}$.

Therefore the semicanonical pseudo-K\"{a}hler structure is as follows:
$$
J_3 =  \left[ \begin {array}{cccccc} \psi_{11}&\psi_{12}&0&0&0&0 \\
\noalign{\medskip}-{\frac {{\psi_{11}}^{2}+1}{\psi_{12}}}&-
\psi_{11}&0&0&0&0\\
 \noalign{\medskip}0&0&\psi_{11}&2\,\psi_{12}/3&0&0\\
 \noalign{\medskip}0&0&-3\,{\frac {{\psi_{11}}^{2}+1}{2\,
\psi_{12}}}&-\psi_{11}&0&0\\ \noalign{\medskip}0&0&-3\,{\frac {{\psi_{11}}^{2}+1}{
\psi_{12}}}&-4\,\psi_{11}&\psi_{11}&2\,\psi_{12}\\
\noalign{\medskip}0&0&0&{\frac {{\psi_{11}}^{2}+1}{\psi_{12}}}
&-{\frac {{\psi_{11}}^{2}+1}{2\,\psi_{12}}}&-\psi_{11}
\end {array} \right],
$$
$$
g_3= \left[ \begin {array}{cccccc} 0&0&0&{\frac {{\psi_{11}}^{2}+1}{
\psi_{12}}}&-{\frac {{\psi_{11}}^{2}+1}{2\,\psi_{12}}} & -\psi_{11}\\ \noalign{\medskip}0&0&0&\psi_{11}&-\psi_{11}/2&
-\psi_{12}\\
\noalign{\medskip}0&0&-3\,{\frac {{\psi_{11}}^{2}+1}{4\,
\psi_{12}}}&-\psi_{11}/2&0&0\\
 \noalign{\medskip}{\frac {{\psi_{11}}^{2}+1}{
\psi_{12}}}&\psi_{11}&-1/2\,\psi_{11}&-\psi_{12}/3&0&0\\
 \noalign{\medskip}-{\frac {{\psi_{11}}^{2}+1}{2\,
 \psi_{12}}}&-\psi_{11}/2&0&0&0&0\\
 \noalign{\medskip}-\psi_{11}& -\psi_{12}&0&0&0&0\end {array} \right]
 $$
On omitting the index of the curvature tensor, it turns out that there is only one (within symmetries) non-zero component of the tensor of curvature  $R_{1, 2, 1, 2} =\frac {2\psi_{12}}{3}$. Then, supposing $\psi_{12}= -a\ne 0$ and $\psi_{11}=0$, we find the following canonical complex structure and pseudo-K\"{a}hler metric of curvature $R_{1, 2, 1, 2}=-\frac {2\,a}{3}$ on a Lie algebra $\mathfrak{h}_{13}$ with $J_3$-invariant 2-planes $\{e_1,e_2\}$ and $\{e_5,e_6\}$:

$J_3(e_2) = -a\, e_1$,\
$J_3(e_3) = \frac {3}{2\,a}\, e_4 +\frac{3}{a}\, e_5$, \
$J_3(e_4) = -\frac{2a}{3}\, e_3 -\frac{1}{a}\, e_6$, \
$J_3(e_6) = -2\,a\, e_5$,

$$
g_3=\left[ \begin {array}{cccccc}
0&0&0&-\frac 1a&\frac{1}{2a}&0\\
0&0&0&0&0&a\\
0&0&\frac{3}{4a}&0&0&0\\
-\frac 1a&0&0&\frac{a}{3}&0&0\\
\frac{1}{2a}&0&0&0&0&0\\
0&a&0&0&0&0\end {array} \right].
$$

\subsection{The Lie group $G_{15}$}
The Lie algebra $\mathfrak{g}_{15}$ is defined in \cite{Goze-Khakim-Med}:
$[X_{1},X_{2}]=X_{4}$,\ $[X_{1},X_{4}]=X_{6}$, \ $[X_{2},X_{3}]=X_{5}$. The Lie algebra has two symplectic structures:

$\omega _{1}=-\alpha^{1}\wedge \alpha^{5}+\alpha^{1}\wedge \alpha^{6}+\alpha^{2}\wedge \alpha^{5}+\alpha^{3}\wedge \alpha^{4},$

$\omega _{2}=\alpha^{1}\wedge \alpha^{6}+\alpha^{2}\wedge \alpha^{4}+\alpha^{3}\wedge \alpha^{5}.$

This is a Lie algebra $M7$ for which Magnin, in \cite{Mag-3}, discovered the complex structures in an explicit form. To use these outcomes, we will make the replacement: $X_1:=e_2$,\  $X_2:=-e_1$,\ $X_5:=-e_6$,\ $X_6:=e_5$. Then the Lie bracket relations and symplectic structures become:
\vspace{1mm}

$[e_1,e_2]=e_{4}$,\ $[e_2,e_{4}]=e_5$, \ $[e_1,e_{3}]=e_6$,
\vspace{1mm}

$\omega _{1}=e^1\wedge e^6+ e^2\wedge e^6+e^2\wedge e^5 +e^{3}\wedge e^{4},$

$\omega _{2}=-e^1\wedge e^{4}+e^2\wedge e^5-e^{3}\wedge e^6.$

\vspace{1mm}
\textbf{First case.}
The symplectic structure is: $\omega_1 =  e^2\wedge e^6 +e^2\wedge e^5 +e^1\wedge e^6 +e^3\wedge e^4$. For any compatible complex structure, the curvature tensor of the associated metric depends on two parameters $\psi_{11}$ and $\psi_{12}\ne 0$. After omitting the index there is one non-zero component:
\[
R_{1, 2, 1, 2} =-\frac{\psi_{11}^4+\psi_{11}^3\,\psi_{12} +2\,\psi_{11}^2-2\,\psi_{12}^2\,\psi_{11}^2+ \psi_{11}\,\psi_{12}+1}{\psi_{12}\,(-2\,\psi_{11}\psi_{12}+1+\psi_{11}^2)}.
\]

Setting the remaining free parameters $\psi_{ij}$ to zero, we find the canonical complex structure $J_1$:

$J_1(e_1) = \psi_{11}\, e_1 -{\frac {1+{\psi_{11}}^{2}}{\psi_{12}}}\, e_2$,

$J_1(e_2) = \psi_{12}\, e_1 -\psi_{11}\, e_2$,

$J_1(e_3) = -{\frac {{\psi_{11}}^{3}-
{\psi_{11}}^{2}\psi_{12}+\psi_{11}+\psi_{12}}{-2\,\psi_{11}
\,\psi_{12}+1+{\psi_{11}}^{2}}}\, e_3 -{\frac { \left( {\psi_{12}}^{2}-2\,\psi_{11}\,\psi_{12}+1+{\psi_{11}}^
{2} \right) \psi_{12}}{-2\,\psi_{11}\,\psi_{12}+1+{\psi_{11}}^
{2}}}\, e_4$,

$J_1(e_4) = \frac {1+2\,{\psi_{11}}^{2}
+{\psi_{11}}^{4}}{\psi_{12}\, (-2\,\psi_{11}\,\psi_{12}+1+{
\psi_{11}}^{2})}\, e_3 +\frac {{\psi_{11}}^{3}-{\psi_{11}}^{2}\psi_{12}+\psi_{11}+
\psi_{12}}{-2\,\psi_{11}\,\psi_{12}+1+{\psi_{11}}^{2}}\, e_4$,

$J_1(e_5) = -\frac {-\psi_{11}\,\psi_{12}+1+
{\psi_{11}}^{2}}{\psi_{12}}\, e_5 +\frac {1+{\psi_{11}}^{2}}{\psi_{12}}
\, e_6$,

$J_1(e_6) = -{\frac {{\psi_{12}}^{2}-2\,
\psi_{11}\,\psi_{12}+1+{\psi_{11}}^{2}}{\psi_{12}}}\, e_5 +{\frac {-\psi_{11}\,\psi_{12}+1+{\psi_{11}}^{2}}{\psi_{12}}}
\, e_6$.
\vspace{2mm}\\
The metric tensor of the pseudo-K\"{a}hler structure is easily found using the formula
$g_1=\omega_1\circ J_1$.

\vspace{1mm}
\textbf{Second case.}
The symplectic structure $\omega _{2}=-e^1\wedge e^{4}+e^2\wedge e^5-e^{3}\wedge e^6$ does not admit a compatible complex structure.


\subsection{The Lie group $G_{11}$}
The Lie algebra $\mathfrak{g}_{11}$ is defined by:

$[e_1,e_2] = e_4$,\ $[e_1,e_4] = e_5$,\ $[e_2,e_3] = e_6$,\ $[e_2,e_4] = e_6$.\\
Its symplectic structure comes from the list in \cite{Goze-Khakim-Med}:

$\omega = e^1\wedge e^6 + e^2\wedge e^5 - e^3\wedge e^4 + \lambda e^2\wedge e^6.$

This is the Lie algebra $M8$, considered in the work of Magnin \cite{Mag-3} and for which he found the complex structures in an explicit form.
There is a multiparametrical set of the compatible complex structures.
For any compatible complex structure $J$ and its associated metric $g=\omega \cdot J$, the curvature tensor depends on three parameters $\lambda$, $\psi_{11}$ and $\psi_{12}\ne 0$.
The canonical compatible complex structure has $J$-invariant 2-plane $\{e_1,e_2\}$, $\{e_3,e_4\}$ and $\{e_5,e_6\}$, however its shape is complicated:

$J(e_1) = \psi_{11}\, e_1 -{\frac {1+{\psi_{11}}^{2}}{\psi_{12}}}\, e_2,\quad
J(e_2)=\psi_{12}\, e_1 -\psi_{11} \, e_2$,

$J(e_3)=-{\frac {2\,{\psi_{12}}^{2}-3\,\psi_{11}\,\lambda\,\psi_{12}
+{\lambda}^{2}(1+{\psi_{11}}^{2})}{\lambda \,\psi_{12}}}\, e_3+{\frac {{\psi_{12}}^{2}-2\,\psi_{11}\,\lambda\,\psi_{12}
+\lambda^{2}(1+\psi_{11}^{2})}{\lambda\, \psi_{12}}}\, e_4$,

$J(e_4)=-{\frac {{\lambda}^{2}({\psi_{11}}^{2}+1)-4\,\psi_{11}\,\lambda\,\psi_{12}
+4\,{\psi_{12}}^{2}}{\psi_{12}\,\lambda}}\, e_3+{\frac {{
\lambda}^{2}({\psi_{11}}^{2}+1)-3\,\psi_{11}\,\lambda\,\psi_{12}+2\,
{\psi_{12}}^{2}}{\psi_{12}\,\lambda}}\, e_4$,

$J(e_5)={\frac {\psi_{11}\,\psi_{12}-
\lambda(1+\psi_{11}^{2})}{\psi_{12}}}\,e_5 +{\frac {1+
{\psi_{11}}^{2}}{\psi_{12}}}\, e_6, $

$J(e_6)=-\frac {{\psi_{12}}
^{2}-2\,\psi_{11}\,\lambda\,\psi_{12}+{\lambda}^{2}+{\lambda}^{2}{
\psi_{11}}^{2}}{\psi_{12}}\,e_5 +\frac {-\psi_{11}\,\psi_{12}+
\lambda+{\psi_{11}}^{2}\lambda}{\psi_{12}}\, e_6. $
\vspace{1mm}

The curvature tensor of associated metric $g=\omega \cdot J$ has the following non-zero component:
$$
R_{1, 2, 1, 2} = -\frac{\lambda^2(\psi_{11}^2+1) -5\lambda \psi_{12}\,\psi_{11} +4\psi_{12}^2}{\lambda \psi_{12}}.
$$


\subsection{The Lie group $G_{10}$}
The Lie algebra $\mathfrak{g}_{10}$ is defined by:

$[e_1,e_2] = e_4$,\ $[e_1,e_4] = e_5$,\ $[e_1,e_3] = e_6$,\ $[e_2,e_4] = e_6$.\\
Its symplectic structure comes from the list in \cite{Goze-Khakim-Med}:

$\omega = e^1\wedge e^6 + e^2\wedge e^5 - e^2\wedge e^6 - e^3\wedge e^4 .$\\

For any compatible complex structure $J$ and its associated metric $g=\omega \cdot J$, the curvature tensor depends on two parameters $\psi_{11}$ and $\psi_{12}\ne 0$.
The canonical compatible complex structure has $J$-invariant 2-plane $\{e_1,e_2\}$, $\{e_3,e_4\}$ and $\{e_5,e_6\}$, however its shape is complicated:

$J(e_1) = \psi_{11}\, e_1 -{\frac {1+{\psi_{11}}^{2}}{\psi_{12}}}\, e_2,\quad
J(e_2)=\psi_{12}\, e_1 -\psi_{11} \, e_2$,

$J(e_3)=-{\frac {\psi_{12}+3\,{
\psi_{11}}^{2}\psi_{12}+2\,{\psi_{12}}^{2}\psi_{11}+\psi_{11}
+{\psi_{11}}^{3}}{2\,{\psi_{12}}^{2}+2\,\psi_{12}\,\psi_{11}+1
+{\psi_{11}}^{2}}}\, e_3 -{\frac {({\psi_{12}}^{2}+2\,\psi_{12}\,\psi_{11}+1+{{
\psi_{11}}}^{2}) \psi_{12}}{2\,{\psi_{12}}^{2}+2\,\psi_{12}\,\psi_{11}+1+{\psi_{11}}^{2}}}\, e_4$,

$J(e_4)={\frac { \left( {\psi_{11}}^{2}+4\,\psi_{12}
\,\psi_{11}+4\,{\psi_{12}}^{2}+1 \right)  \left( 1+{\psi_{11}}^{
2} \right) }{\psi_{12}\, \left( 2\,{\psi_{12}}^{2}+2\,\psi_{12}
\,\psi_{11}+1+{\psi_{11}}^{2} \right) }}\, e_3 +{\frac {\psi_{12}+3\,{\psi_{11}}^{2}\psi_{12}+2\,{\psi_{12}}^{2}\psi_{11}+\psi_{11}
+{\psi_{11}}^{3}}{2\,{\psi_{12}}^{2}+2\,\psi_{12}\,\psi_{11}+1
+{\psi_{11}}^{2}}}\, e_4$,

$J(e_5)={\frac {\psi_{12}\,\psi_{11}+1+{\psi_{11}}^{2}}{\psi_{12}}}\,e_5 +{\frac {1+\psi_{11}^{2}}{\psi_{12}}}\, e_6$.

$J(e_6)=-{\frac {{\psi_{12}}^{2}+2\,\psi_{12}\,\psi_{11}+1+{\psi_{11}}^{2}}{\psi_{12}}}\,e_5 -{\frac {\psi_{12}\,\psi_{11}+1+{\psi_{11}}^{2}}{\psi_{12}}}\, e_6$.
\vspace{1mm}

The curvature tensor of associated metric $g=\omega \cdot J$ has the following non-zero component:
$$
R_{1, 2, 1, 2} =
{\frac {{\psi_{11}}^{4}+4\,\psi_{12}\,{\psi_{11}}^{3}+3\,
{\psi_{11}}^{2}{\psi_{12}}^{2}+2\,{\psi_{11}}^{2}+4\,\psi_{11}\,
\psi_{12}-2\,\psi_{11}\,{\psi_{12}}^{3}+3\,{\psi_{12}}^{2}+1-2
\,{\psi_{12}}^{4}}{\psi_{12}\, \left( 2\,{\psi_{12}}^{2}+2\,
\psi_{11}\,\psi_{12}+1+{\psi_{11}}^{2} \right) }}.
$$

\subsection{The Lie group $G_{12}$}
The Lie algebra $\mathfrak{h}_{12}$ is defined by: $[X_1,X_2] = X_4$, \ $[X_1,X_4] = X_5$, \ $[X_1,X_3] = X_6$, \ $[X_2,X_3] = -X_5$, \ $[X_2,X_4] = X_6$.
Its symplectic structure comes from the list in \cite{Goze-Khakim-Med}:

$\omega = \lambda \alpha^1\wedge \alpha^5 +\alpha^2\wedge \alpha^6 +(\lambda+1)\alpha^3\wedge \alpha^4, \qquad  \lambda\ne 0; -1.$

This is the Lie algebra $M10$, considered in the work of Magnin \cite{Mag-3} and for which he found the complex structures in an explicit form. To use these outcomes, we will make the replacements: $X_1=-e_1$, $X_2=e_2$, $X_3=-e_4$, $X_4=-e_3$, $X_5=e_5$, $X_6=e_6$.
Then the Lie bracket relations and symplectic structures become:
\vspace{1mm}

$[e_1,e_2] = e_3$,\ $[e_1,e_3] = e_5$,\ $[e_1,e_4] = e_6$,\ $[e_2,e_4] = e_5$,\ $[e_2,e_3] = -e_6$.
\vspace{1mm}

$\omega = -\lambda e^1\wedge e^5 +e^2\wedge e^6 -(\lambda+1)e^3\wedge e^4, \qquad  \lambda\ne 0; -1.$
\vspace{1mm}

In \cite{Mag-3} it is shown that on the given group there are several types of complex structures from which we will choose the ones that are compatible with $\omega$.

\vspace{1mm}
\textbf{First type: $\psi_{12}\ne \psi_{34}$.}\\
For any compatible complex structure of this type and its associated metric, the curvature tensor depends on three parameters $\psi_{11}$, $\psi_{33}$ and $\psi_{12}\ne 0$.
The elements of the tensor of curvature have very complicated expressions which also depend on the parameter $\lambda$.
We will reduce the expressions for the compatible complex structure $J_1$ and the corresponding metric $g_1$ using the zero parameters which do not appear in the curvature and adding the side condition: $\psi_{11}=0$, $\psi_{33}=0$.
\[
J_1(e_2)=\psi_{12}\, e_1,\qquad J_1(e_4)={\frac { \left( \lambda-1 \right)
\psi_{12}}{{\lambda\,\psi_{12}}^{2}-1}}\, e_3,\qquad J_1(e_6)=-{\frac {1}{\lambda\,\psi_{12}}}\, e_5.
\]
$$
g_1= \left[ \begin {array}{cccccc} 0&0&0&0&0&{\psi_{12}}^{-1}\\ \noalign{\medskip}0&0&0&0&\lambda\,\psi_{12}&0
\\ \noalign{\medskip}0&0&-{\frac { \left( \lambda+1 \right)  \left(
{\psi_{12}}^{2}\lambda-1 \right) }{ \left( \lambda-1 \right)
\psi_{12}}}&0&0&0\\ \noalign{\medskip}0&0&0&-{\frac {\psi_{12}\,
 \left( {\lambda}^{2}-1 \right) }{{\psi_{12}}^{2}\lambda-1}}&0&0
\\ \noalign{\medskip}0&\lambda\,\psi_{12}&0&0&0&0
\\ \noalign{\medskip}{\psi_{12}}^{-1}&0&0&0&0&0\end {array} \right]
$$

The curvature tensor has the following non-zero components:

$$
R_{1, 2, 2}^6  = -\frac{\lambda^3\psi_{12}^2+3\lambda^2\psi_{12}^2 -3\lambda - 1}{\lambda^2-1},\quad
R_{1, 2, 1}^5  = \frac{\lambda^3\psi_{12}^2+3\lambda^2\psi_{12}^2 -3\lambda - 1}{\lambda(\lambda^2-1)\psi_{12}^2}.
$$

After omitting the index, there remains one component:
$$
R_{1, 2, 1, 2}=\frac{\lambda^3\psi_{12}^2+3\lambda^2\psi_{12}^2 -3\lambda - 1}{(\lambda^2-1)\psi_{12}^2}.
$$

\textbf{Second type: $\psi_{12}=\psi_{34}$, $\psi_{11}=\psi_{33}=0$.}\\
The compatible complex structure of this type is of the form:
$$
J_2= \left[ \begin {array}{cccccc} 0&-1&0&0&0&0\\ \noalign{\medskip}1&0&0&0
&0&0\\ \noalign{\medskip}\psi_{31}&\psi_{41}&0&-1&0&0
\\ \noalign{\medskip}\psi_{41}&-\psi_{31}&1&0&0&0
\\ \noalign{\medskip}\psi_{51}&\psi_{52}&{\frac {\psi_{41}\,
 \left( \lambda+1 \right) }{\lambda}}&-{\frac {\psi_{31}\, \left(
\lambda+1 \right) }{\lambda}}&0&{\lambda}^{-1}\\
\noalign{\medskip}-
\lambda\,\psi_{52}&J_{61}&\psi_{31}
\,(\lambda+1)&\psi_{41}\,(\lambda+1)&-\lambda&0
\end {array} \right],
$$
where $J_{61}=\lambda\,\psi_{51}-({\psi_{41}}^{2}+{\psi_{31}}^{2})(\lambda
+1)$.

For any such complex structure and its associated metric, the curvature tensor does not depend on the parameters $\psi_{ij}$, and depends only on $\lambda$.
The non-zero elements of the tensor of curvature are:  $R_{1, 2, 2}^6  = -\frac{\lambda^2+4\lambda +1}{\lambda+1}$,\
$R_{1, 2, 1}^5  = \frac{\lambda^2+4\lambda +1}{\lambda(\lambda+1)}$.
After omitting the index, there remains one component: $R_{1, 2, 1, 2}=-\frac{\lambda^2+4\lambda +1}{\lambda+1}$.
From our point of view, it is possible to consider that all the remaining parameters are zero. Then we find the canonical expressions for the complex structure and the pseudo-Riemannian metric:
$$
J_{20}(e_1)=e_2,\quad J_{20}(e_3)=e_4,\quad J_{20}(e_5)=-\lambda\, e_6,
$$
$$
g_{20}=-2\,e^1\cdot e^6 -2\,\lambda\,e^2\cdot e^5 -(\lambda+1)\, (e^3)^2 -(\lambda+1)\, (e^4)^2.
$$

\textbf{Third type: $\psi_{34}=\psi_{12}$, $\psi_{33}\ne \psi_{11}$.}\\
For any compatible complex structure of this type and the associated metric, the curvature tensor does not depend on the parameters $\psi_{ij}$. Converting the free parameters to zero, we find the following canonical expressions for the complex structure and the associated pseudo-K\"{a}hler metric:
$$
J_{30}= \left[ \begin {array}{cccccc} 1&\sqrt {2}&0&0&0&0\\ \noalign{\medskip}-\sqrt {2}&-1&0&0&0&0\\ \noalign{\medskip}0&0&-{
\frac {\lambda+1}{\lambda-1}}&\sqrt {2}&0&0\\ \noalign{\medskip}0&0&-{
\frac { \left( {\lambda}^{2}+1 \right) \sqrt {2}}{ \left( \lambda-1
 \right) ^{2}}}&{\frac {\lambda+1}{\lambda-1}}&0&0
\\ \noalign{\medskip}0&0&0&0&-1&-{\frac {\sqrt {2}}{\lambda}}
\\ \noalign{\medskip}0&0&0&0&\sqrt {2}\lambda&1\end {array} \right],
$$
$$
g_{30}= \left[ \begin {array}{cccccc} 0&0&0&0&\lambda&\sqrt {2}\\ \noalign{\medskip}0&0&0&0&\sqrt {2}\lambda&1\\ \noalign{\medskip}0&0
&{\frac { \left( \lambda+1 \right)  \left( {\lambda}^{2}+1 \right)
\sqrt {2}}{ \left( \lambda-1 \right) ^{2}}}&-{\frac { \left( \lambda+1
 \right) ^{2}}{\lambda-1}}&0&0\\ \noalign{\medskip}0&0&-{\frac {
 \left( \lambda+1 \right) ^{2}}{\lambda-1}}& \left( \lambda+1 \right)
\sqrt {2}&0&0\\ \noalign{\medskip}\lambda&\sqrt {2}\lambda&0&0&0&0
\\ \noalign{\medskip}\sqrt {2}&1&0&0&0&0\end {array} \right].
$$
The curvature tensor has the following non-zero components:
\[
R_{1, 2, 2}^6  = -\frac{2(\lambda^4+4\lambda^3 -2\lambda^2 +4\lambda +1)}{(\lambda-1)^2(\lambda+1)},\quad
R_{1, 2, 1}^5  = \frac{2(\lambda^4+4\lambda^3 -2\lambda^2 +4\lambda +1)}{(\lambda+1)\lambda(\lambda-1)^2}.
\]
After omitting the index, there remains one component:
\[
R_{1, 2, 1, 2}=\frac{\sqrt{2}(\lambda^4+4\lambda^3 -2\lambda^2 +4\lambda +1)}{(\lambda-1)^2(\lambda+1)}.
\]


\section{Lie algebras of type $(2,6)$}

There are three Lie algebras of type $(2,6)$ admitting a pseudo-K\"{a}hler structure.

\subsection{The Lie group $G_{24}$}
The Lie algebra $\mathfrak{h}_{24}$ is defined by:

$[e_1,e_4] = e_6$, $[e_2,e_3] = e_5$.\\
This Lie algebra is a direct product two three-dimensional Lie algebras of Heisenberg: $\mathfrak{g}_{24} =\mathfrak{h}_3\times\mathfrak{h}_3$
The symplectic structure is:

$\omega = e^1\wedge e^6 +e^2\wedge e^5 +e^3\wedge e^4$.

For any compatible complex structure $J$ and its associated metric $g=\omega \circ J$, the curvature tensor depends on two parameters $\psi_{11}\ne 0$ and $\psi_{12}\ne 0$.
Therefore the semicanonical pseudo-K\"{a}hler structure is as follows:

$J(e_1) = \psi_{11}\, e_1 -\frac {1+\psi_{11}^{2}}{\psi_{12}}\, e_2,\quad J(e_2) = \psi_{12}\, e_1 -\psi_{11}\, e_2$,

$J(e_3) = {\frac {{\psi_{11}}^{2}-1}{2\psi_{11}}}\, e_3 -{\frac {2\,{\psi_{11}}^{2}+ {\psi_{11}}^{4}+1}{2\psi_{11}\,{\psi_{12}}^{2}}}\, e_4,\quad J(e_4) = \frac {\psi_{12}^{2}}{2\psi_{11}}\, e_3 -\frac {\psi_{11}^{2}-1}{2\,\psi_{11}}\, e_4$,

$J(e_5) = \psi_{11}\, e_5 +\frac {1+\psi_{11}^{2}}{\psi_{12}}\, e_6,\quad
J(e_6) = -\psi_{12}\, e_5 -\psi_{11}\, e_6.$

$$
g= \left[ \begin {array}{cccccc}
0&0&0&0&{\frac {1+{\psi_{11}}^{2}}{\psi_{12}}}&-\psi_{11}\\
\noalign{\medskip}0&0&0&0&\psi_{11}&-\psi_{12}\\
\noalign{\medskip}0&0&-{\frac {2\,{\psi_{11}}^{2}+
{\psi_{11}}^{4}+1}{2\,{\psi_{12}}^{2}\psi_{11}}}&-{\frac {{
\psi_{11}}^{2}-1}{2\,\psi_{11}}}&0&0\\
\noalign{\medskip}0&0&-{\frac {{\psi_{11}}^{2}-1}{2\,\psi_{11}}}&-{\frac {{\psi_{12}}^{2}}{2\,\psi_{11}}}&0&0\\
\noalign{\medskip}{\frac {1+{\psi_{11}}^{2}}{\psi_{12}}}&\psi_{11}&0&0&0&0\\
\noalign{\medskip}-\psi_{11}&-\psi_{12}&0&0&0&0\end {array} \right].
$$

For any compatible complex structure and its associated metric, the curvature tensor depends on two parameters $\psi_{11}\ne 0$ è $\psi_{12}\ne 0$: $R_{1, 2, 2}^6 = -\psi_{11}^2$,\
$R_{1, 2, 1}^6 =-\frac{(1 +\psi_{11}^2)\psi_{11}}{\psi_{12}}$,\
$R_{1, 2, 2}^5 =-\psi_{12} \psi_{11}$,\
$R_{1, 2, 1}^5 = -\psi_{11}^2$.\
On omitting the index of the curvature tensor, it turns out that there is only one (within symmetries) non-zero component of the tensor of curvature
$R_{1, 2, 1, 2} =  \psi_{11}$.

\subsection{The Lie group $G_{17}$}
The Lie algebra $\mathfrak{h}_{17}$ is defined by:

$[e^1,e^3] = e^5$, $[e^1,e^4] = e^6$, $[e^2,e^3] = e^6$,\\
Its symplectic structure comes from the list in \cite{Goze-Khakim-Med}:

$\omega = e^1\wedge e^6 + e^2\wedge e^5 + e^3\wedge e^4$

For any compatible complex structure and its associated metric, the curvature tensor depends on two parameters $\psi_{11}$ and $\psi_{12}\ne 0$:

Setting the remaining free parameters $\psi_{ij}$ to zero, we find the canonical complex structure $J$ and associated metric $g=\omega \circ J$:

$J(e_1) = \psi_{11}\, e_1 -{\frac {1+{\psi_{11}}^{2}}{\psi_{12}}}\, e_2,\qquad
J(e_2) = \psi_{12}\, e_1 -\psi_{11}\, e_2$,\

$J(e_3) = \psi_{11}\,e_3 +2\,{\frac {1+{\psi_{11}}^{2}}{\psi_{12}}}\, e_4,\qquad J(e_4) = -\frac {\psi_{12}}{2}\, e_3 -\psi_{11}\, e_4$,\

$J(e_5) = \psi_{11}\, e_5 +{\frac {1+{\psi_{11}}^{2}}{\psi_{12}}}\, e_6,\qquad
J(e_6) = -\psi_{12}\, e_5 -\psi_{11}\, e_6$.
$$
g=\left[ \begin {array}{cccccc} 0&0&0&0&{\frac {1+{\psi_{11}}^{2}}{
\psi_{12}}}&-\psi_{11}\\
\noalign{\medskip}0&0&0&0&\psi_{11}&-\psi_{12}\\
\noalign{\medskip}0&0&2\,{\frac {1+{\psi_{11}}^{2}}{\psi_{12}}}&-\psi_{11}&0&0\\ \noalign{\medskip}0&0&-\psi_{11}&\psi_{12}/2&0&0\\
\noalign{\medskip}{\frac {1+{\psi_{11}}^{2}}{\psi_{12}}}&\psi_{11}&0&0&0&0\\ \noalign{\medskip}-\psi_{11}&-\psi_{12}&0&0&0&0\end {array} \right]
$$
The curvature tensor has the following non-zero components:
$R_{1, 2, 1}^5 = \psi_{11}\psi_{12}$,\
$R_{1, 2, 1}^6 = 1+\psi_{11}^2$,\
$R_{1, 2, 2}^5 = \psi_{12}^2$,\
$R_{1, 2, 2}^6 = \psi_{11}\psi_{12} $.
After omitting the index, there remains one component: $R_{1, 2, 1, 2} = -\psi_{12}$.

\subsection{The Lie group $G_{16}$}
The Lie algebra $\mathfrak{g}_{16}$ is defined by: $[X_1,X_2] = X_5$, $[X_1,X_3] = X_6$, $[X_2,X_4] = X_6$, $[X_3,X_4] =-X_5$. It is complex Heisenberg Lie algebra. The algebra $\mathfrak{g}_{16}$ has (see \cite{Goze-Khakim-Med}) two symplectic structures which are noted using the dual base $\alpha^i$ as follows:

$\omega_1 = \alpha^1\wedge \alpha^6 +\alpha^2\wedge \alpha^3 -\alpha^4\wedge \alpha^5$,

$\omega_2 = \alpha^1\wedge \alpha^6 -\alpha^2\wedge \alpha^3 +\alpha^4\wedge \alpha^5$.\\
Left-invariant complex structures on this group are discovered in an explicit form in the work of Magnin \cite{Mag-3} (algebra $M5$). We will use Magnin's results, so we will make the replacement: $X_1=e_1$, $X_2=e_3$, $X_3=e_4$, $X_4=e_2$, $X_5=e_5$, $X_6=e_6$. Then the non-trivial brackets are given in \cite{Mag-3}:

$[e_1,e_3] = e_5$, $[e_1,e_4] = e_6$, $[e_2,e_3] =-e_6$, $[e_2,e_4] =e_5$.\\
The symplectic forms become:

$\omega_1 = e^1\wedge e^6 +e^3\wedge e^4 -e^2\wedge e^5$,

$\omega_2 = e^1\wedge e^6 -e^3\wedge e^4 +e^2\wedge e^5$.

Let $z_1=\frac{1}{2}(e_1+ie_2)$, $z_2=\frac{1}{2}(e_3-ie_4)$, $z_3=\frac{1}{2}(e_5-ie_6)$, then Lie algebra $\mathfrak{g}_{16}$ is defined by: $[z_1,z_2]=z_3$.
If $J_0$ is complex structure of the complex Lie algebra $M5$, then: $J_0(e_1)=-e_2$, $J_0(e_3)=e_4$, $J_0(e_5)=e_6$.
The complex structure $J_0$ is not $\omega_1$-compatible, but compatible with $\omega_2$. Thus, the metric $g_0$ of pseudo-K\"{a}hler structure $(J_0,\omega_2,g_0)$ is as follows:
$$
g_{0}=2\,e^1\cdot e^6 -2\,e^2\cdot e^5 -\, (e^3)^2 -\, (e^4)^2.
$$
\noindent
The curvature tensor has the following non-zero components:
$R_{1, 2, 1}^6 =1$,\ $R_{1, 2, 2}^5 =1$,\ $R_{1, 2, 1, 2} =-1$.

 \vspace{1.5mm}
\textbf{First case.} The symplectic structure is:
$\omega_1 =  e^1\wedge e^6 -e^2\wedge e^5 +e^3\wedge e^4$.

For any compatible complex structure $J$, the curvature tensor of the associated metric $g=\omega_1 \circ J$ depends on two parameters $\psi_{11}$ and $\psi_{12}\ne 0$.
Setting the remaining free parameters $\psi_{ij}$ to zero, we find the canonical complex structure $J_1$ and pseudo-K\"{a}hler metric $g_1=\omega_1 \circ J_1$:

$J_1(e_1) = \psi_{11}\, e_1 -{\frac {1+{\psi_{11}}^{2}}{\psi_{12}}}\, e_2,\qquad
J_1(e_2) = \psi_{12}\, e_1 -\psi_{11}\, e_2$,\

$J_1(e_3) = -{\frac {\psi_{11}\, (1+{\psi_{11}}^{2}-{\psi_{12}}^{2})}{ {\psi_{12}}^{2}+1+{\psi_{11}}^{2}}}\, e_3 -{\frac {2( 1+{\psi_{11}}^{2}) \psi_{12}}{{\psi_{12}}^{2}+1+{\psi_{11}}^{2}}}\, e_4$,\

$J_1(e_4) = {\frac {2\,{\psi_{11}}^{2} +{\psi_{11}}^{4}-2\,{\psi_{11}}^{2}{\psi_{12}}^{2} +{\psi_{12}}^{4}+2
\,{\psi_{12}}^{2}+1}{2\psi_{12}\, ({\psi_{12}}^{2} +1
+{\psi_{11}}^{2}) }}\, e_3 +{\frac {\psi_{11}\, (1+{\psi_{11}}^{2}-{\psi_{12}}^{2}) }{{\psi_{12}}^{2}+1+{\psi_{11}}^{2}}}\, e_4$,\

$J_1(e_5) = \psi_{11}\, e_5 -{\frac {1+{\psi_{11}}^{2}}{\psi_{12}}}\, e_6,\qquad
J_1(e_6) = \psi_{12}\, e_5 -\psi_{11}\, e_6$.
 \vspace{1.5mm}\\
The curvature tensor has the following non-zero components:
\[
R_{1, 2, 1}^5 =\frac{\psi_{11}(1+\psi_{11}^2+\psi_{12}^2)}{\psi_{12}},\quad
R_{1, 2, 2}^5=1+\psi_{11}^2+\psi_{12}^2,
\]
\[
R_{1, 2, 1}^6 =-\frac{(1+\psi_{11}^2)(1+\psi_{11}^2+\psi_{12}^2)}{\psi_{12}^2},\quad
R_{1, 2, 2}^6=-\frac{\psi_{11}(1+\psi_{11}^2+\psi_{12}^2)}{\psi_{12}}.
\]
After omitting the index, there remains one component:
\[
R_{1, 2, 1, 2} =\frac{1+\psi_{11}^2+\psi_{12}^2}{\psi_{12}}.
\]

 \vspace{1.5mm}
\textbf{Second case.} The symplectic structure is:
$\omega_2 =  e^1\wedge e^6 +e^2\wedge e^5 -e^3\wedge e^4$.
For any compatible complex structure $J$, the curvature tensor of the associated metric $g=\omega_2 \circ J$ depends on two parameters $\psi_{34}$ and $\psi_{12}=\pm 1$. Let $\psi_{12}=1$.
Setting the remaining free parameters $\psi_{ij}$ to zero, we find the canonical complex structure $J_2$ and pseudo-K\"{a}hler metric $g_2=\omega_2 \circ J_2$:
$$
J_2(e_2) = e_1,\quad J_2(e_4) = \psi_{34}\, e_3,\quad
J_2(e_6) = -e_5,
$$
$$
g_{2}=2\,e^1\cdot e^5 -2\,e^2\cdot e^6 +\frac{1}{\psi_{34}}\, (e^3)^2 +\psi_{34}\, (e^4)^2.
$$

The curvature tensor has the following non-zero components:
$R_{1, 2, 1}^6 = -{\rm sign}(\psi_{12})\psi_{34}$,\
$R_{1, 2, 2}^5 = -{\rm sign}(\psi_{12})\psi_{34}$. After omitting the index, there remains one component:  $R_{1, 2,1, 2} = {\rm sign}(\psi_{12})\psi_{34}$.

The complex structure $J_2$ is biinvariant (then $G_{16}$ is complex Lie group) if and only if $\psi_{34} =-1$, that is $J_2=J_0$.

\section{Lie algebras of type $(4,6)$}
In this class there is only one algebra.

\subsection{The Lie group $G_{25}$}
The Lie algebra $\mathfrak{h}_{25}$ is defined by $[e_1,e_2] =e_3$.
This Lie algebra is a direct product of a three-dimensional nilpotent Lie algebra of Heisenberg $\mathfrak{h}_3$ and $\mathbb{R}^3$.
The symplectic structure is:

$\omega=e^1\wedge e^3+e^2\wedge e^4+e^5\wedge e^6$

In this case there is an 8-parametrical set of the compatible complex structures and pseudo-K\"{a}hler metrics. All the metrics are flat. Therefore we will specify only the simplest expressions, without parameters:
$$
J(e_1) = e_2,\quad J(e_3) = e_4,\quad J(e_5) = e_6.
$$
$$
g=\left[ \begin {array}{cccccc}
0&0&0&-1&0&0\\
0&0&1&0&0&0\\
0&1&0&0&0&0\\
-1&0&0&0&0&0\\
0&0&0&0&1&0\\
0&0&0&0&0&1\end {array} \right].
$$

\section{Lie algebras of type $(3,6)$}
There are two Lie algebras of type (3,6) admitting a pseudo-K\"{a}hler structure.

\subsection{The Lie group $G_{18}$}
The Lie algebra $\mathfrak{h}_{18}$ is defined by:
$[e_{1},e_{2}] =e_{4}$, $[e_{1},e_{3}]=e_{5}$, $[e_{2},e_{3}] =e_{6}$.
This Lie algebra has three symplectic structures \cite{Goze-Khakim-Med}:

$\omega _{1}(\lambda )=e^{1}\wedge e^{6}+\lambda e^{2}\wedge e^{5}+\left( \lambda -1\right) e^{3}\wedge e^{4}$, $\quad \lambda \ne 0,\,1$,

$\omega _{2}(\lambda ) = e^{1}\wedge e^{5}{+}\lambda e^{1}\wedge e^{6}{-}\lambda e^{2}\wedge e^{5}{+}e^{2}\wedge e^{6}{-}2\lambda e^{3}\wedge e^{4}$,
$\quad \lambda \ne 0$,

$\omega _{3}= -e^{1}\wedge e^{6} +e^{2}\wedge e^{5}+2e^{3}\wedge e^{4}+e^{3}\wedge e^{5}$.\\
Left-invariant complex structures on this group are discovered in an explicit form in the work of Magnin \cite{Mag-3} (algebra $M3$).

\textbf{First case.}  The symplectic structure is:
$\omega_1 =  e^1\wedge e^6 +\lambda e^2\wedge e^5 +(\lambda-1) e^3\wedge e^4$.
The compatible complex structures exist only at $\lambda =-1$, therefore

$\omega_1 =  e^1\wedge e^6 - e^2\wedge e^5 -2\, e^3\wedge e^4$.\\
For any compatible complex structure and its associated metric, the curvature tensor depends on three parameters $\psi_{11}$, $\psi_{12}$ and $\psi_{34}$. The curvature tensor has the following non-zero components:
$R_{1, 2, 1}^6 =-\frac{2\psi_{34}(1 + \psi_{11}^2 )}{\psi_{12}}$,\
$R_{1, 2, 2}^6 =-2\psi_{11} \psi_{34}$,\
$R_{1, 2, 2}^5~=~2\psi_{12} \psi_{34}$,\
$R_{1, 2, 1}^5 =2\psi_{11} \psi_{34}$.
After omitting the index, there remains one component: $R_{1, 2, 1, 2} =2\psi_{34}$.
Setting the remaining free parameters $\psi_{ij}$ to zero, we find the canonical complex structure $J_1$ and the pseudo-K\"{a}hler metric $g_1 =\omega_1 \circ J_1$:

$J_1(e_1) = \psi_{11}\, e_1 -{\frac {1+{\psi_{11}}^{2}}{\psi_{12}}}\, e_2,\qquad J_1(e_2) = \psi_{12}\, e_1 -\psi_{11}\, e_2$,\

$J_1(e_4) =  -{\frac {1}{\psi_{34}}}\, e_4,\qquad J_1(e_4) = \psi_{34}\, e_3$,\

$J_1(e_5) = \psi_{11}\, e_5 -{\frac {1+{\psi_{11}}^{2}}{\psi_{12}}}\, e_6,\qquad J_1(e_6) = \psi_{12}\, e_5 -\psi_{11}\, e_6$.
$$
g_1=  \left[ \begin {array}{cccccc}
0&0&0&0&-{\frac {1+{\psi_{11}}^{2}}{\psi_{12}}}&-\psi_{11}\\
 \noalign{\medskip}0&0&0&0&-\psi_{11}&-\psi_{12}\\
 \noalign{\medskip}0&0&{\frac {2}{\psi_{34}}}&0&0&0\\
 \noalign{\medskip}0&0&0&2\,\psi_{34}&0&0\\
 \noalign{\medskip}-{\frac {1+{\psi_{11}}^{2}}{\psi_{12}}}&-\psi_{11}&0&0&0&0\\ \noalign{\medskip}-\psi_{11}&-\psi_{12}&0&0&0&0\end {array} \right]
$$


\textbf{Second case.}  The symplectic structure is: $\omega_2 =  e^1\wedge e^5 +e^2\wedge e^6 +\lambda e^1\wedge e^6 -\lambda e^2\wedge e^5 -\\ -2\lambda e^3\wedge e^4$.
The compatible complex structures exist only for a case $\psi_{16}=0$, $\psi_{25}=0$ and $\psi_{12}^2=1$. We take a case $\psi_{12}=1$.
For any compatible complex structure and its associated metric, the curvature tensor depends on one parameter $\psi_{34}$.
The curvature tensor has the following non-zero components:
$R_{1, 2, 2}^6 =-\frac{2 \lambda \psi_{34}}{\lambda^2 +1}$,\
$R_{1, 2, 1}^6 = -\frac{2 \lambda^2 \psi_{34}}{\lambda^2 +1}$,\
$R_{1, 2, 1}^5 = -\frac{2 \lambda \psi_{34}}{\lambda^2 +1}$,\
$R_{1, 2, 2}^5 =\frac{2 \lambda^2 \psi_{34}}{\lambda^2 +1}$.
After omitting the index, there remains one component:
$R_{1, 2, 1, 2} = 2\lambda \psi_{34}$.\
Setting the remaining free parameters $\psi_{ij}$ to zero, we find the canonical complex structure $J_2$ and the pseudo-K\"{a}hler metric $g_2 =\omega_2 \circ J_2$:
$$
J_2(e_2) =e_1,\quad
J_2(e_4) = \psi_{34}\, e_3,\quad
J_2(e_6) = e_5,
$$
$$
g_2=\left[ \begin {array}{cccccc} 0&0&0&0&-\lambda&1\\
0&0&0&0&-1&-\lambda\\
0&0&\frac {2\,\lambda}{\psi_{34}}&0&0&0\\
0&0&0&2\,\lambda\,\psi_{34}&0&0\\
-\lambda&-1&0&0&0&0\\
1&-\lambda&0&0&0&0\end {array} \right].
$$


\textbf{Third case.}  The symplectic structure is: \\
\centerline{$\omega_3 =  -e^1\wedge e^6 +e^2\wedge e^5 +2e^3\wedge e^4 +e^3\wedge e^5 $.}

For any compatible complex structure and its associated metric, the curvature tensor depends on two parameters $\psi_{25}\ne 0$ and $\psi_{46}\ne 0$. The curvature tensor has the following non-zero components:
$R_{1, 2, 1}^6 = -\frac{6 \psi_{25}}{\psi_{46}}$,\
$R_{1, 2, 2}^5 =  54 \psi_{46} \psi_{25}$,\
$R_{1, 2, 3}^5 = 18 \psi_{46} \psi_{25}$,\
$R_{1, 2, 3}^4 = -6 \psi_{46} \psi_{25}$,\
$R_{1, 3, 1}^6 =-\frac{2 \psi_{25}}{\psi_{46}}$,\
$R_{1, 3, 2}^5 = 18 \psi_{46} \psi_{25}$,\
$R_{1, 2, 2}^4 =-18 \psi_{46} \psi_{25}$,\
$R_{1, 3, 2}^4 = -6 \psi_{46} \psi_{25}$,\
$R_{1, 3, 3}^5 = 6 \psi_{46} \psi_{25}$,\
$R_{1, 3, 3}^4 = -2 \psi_{46} \psi_{25}$.

After omitting the index, there remains three components:
$R_{1, 2, 1, 3} = 6 \psi_{25}$,\
$R_{1, 3, 1, 3} = 2 \psi_{25}$,\
$R_{1, 2, 1, 2} = 18 \psi_{25}$.
Setting the remaining free parameters $\psi_{ij}$ to zero, we find the canonical complex structure $J_3$ and the pseudo-K\"{a}hler metric $g_3 =\omega_2 \circ J_3$:
$$
J_3= \left[ \begin {array}{cccccc} 0&-3\,\psi_{46}&-\psi_{46}&0&0&0\\ \noalign{\medskip}0&0&0&3\,\psi_{25}&\psi_{25}&0\\
\noalign{\medskip}{\psi_{46}}^{-1}&0&0&-9\,\psi_{25}&-3\,\psi_{25}&0\\ \noalign{\medskip}0&-{\psi_{25}}^{-1}&0&0&0&\psi_{46} \\ \noalign{\medskip}0&2\,{\psi_{25}}^{-1}&0&0&0&-3\,\psi_{46}\\ \noalign{\medskip}0&0&0&2\,{\psi_{46}}^{-1}&{\psi_{46}}^{-1}&0
\end {array} \right],
$$
$$
g_{3}=\left[ \begin {array}{cccccc}
0&0&0&-2\,{\psi_{46}}^{-1}&-{\psi_{46}}^{-1}&0\\
 \noalign{\medskip}0&2\,{\psi_{25}}^{-1}&0&0&0&-3\,\psi_{46}\\
 \noalign{\medskip}0&0&0&0&0&-\psi_{46}\\
 \noalign{\medskip}-2\,{\psi_{46}}^{-1}&0&0&18\,\psi_{25}&6\,\psi_{25}&0\\
 \noalign{\medskip}-{\psi_{46}}^{-1}&0&0&6\,\psi_{25}&2\,\psi_{25}&0\\
 \noalign{\medskip}0&-3\,\psi_{46}&-\psi_{46}&0
&0&0\end {array} \right].
$$

\subsection{The Lie group $G_{23}$}
The Lie algebra $\mathfrak{h}_{23}$ is defined by: $[e_1,e_2] = e_5$, $[e_1,e_3] = e_6$.  There are, according to \cite{Goze-Khakim-Med}, three different symplectic structures:

$\omega_1 =  e^1\wedge e^6 +e^2\wedge e^5 +e^3\wedge e^4$,\

$\omega_2 =  e^1\wedge e^4 +e^2\wedge e^6 +e^3\wedge e^5$ and

$\omega_3 =  e^1\wedge e^4 +e^2\wedge e^6 - e^3\wedge e^5$.

For the first two symplectic structures there are no compatible complex structures. For the third symplectic structure $\omega_3$, there is a set of compatible complex structures that depend on several parameters. It will be convenient to enumerate basis vectors as follows: $e_2:=e_1$, $e_3:=-e_2$, $e_1:=e_3$, then

$[e_1,e_3] = -e_5$,\ $[e_2,e_3] = e_6$\\
and

$\omega_3 =  e^1\wedge e^6 + e^2\wedge e^5+ e^3\wedge e^4.$

It is easy to see that the given Lie algebra comes from the $\mathbb{R}^4= \mathbb{R}\{e_1, e_2, e_5, e_6\}$ semidirect product with $\mathbb{R} e_3$ and then a direct product with $\mathbb{R}e_4$,
$\mathfrak{g}_{23}=\mathbb{R}^4 \rtimes \mathbb{R} e_3\times \mathbb{R} e_4$.

The set of complex structures whose parameters influence the curvature operates on the invariant 2-planes $\{e_1, e_2\}$, $\{e_3, e_4\}$ è $\{e_5, e_6\}$  as follows:
$$
J(e_2) = \psi_{12}\, e_1 -\psi_{11}\, e_2,\quad
J(e_4) = \psi_{34}\, e_3,\quad
J(e_6) = -\psi_{12}\, e_5 -\psi_{11}\, e_6.
$$

The curvature tensor depends on three parameters and has the following non-zero components:
$R_{1, 2, 1}^6 =\frac{\psi_{34} (1 + \psi_{11}^2 )}{\psi_{12}}$,\
$R_{1, 2, 1}^5 = \psi_{11} \psi_{34}$,\
$R_{1, 2, 2}^5~=~\psi_{12} \psi_{34}$.
After omitting the index, there remains one component
$R_{1, 2, 1, 2} =-\psi_{34}$.
Supposing $\psi_{34}=-a$, $\psi_{12}=1$, $\psi_{11}=0$ and $\psi_{33}=0$, we find the canonical pseudo-K\"{a}hler structure of curvature $R_{1, 2, 1, 2} =a$:
$$
J(e_2) = e_1,\quad J(e_4) = -a\, e_3 ,\quad J(e_6) = -e_5.
$$
$$
g=\left[ \begin {array}{cccccc}
0&0&0&0&1&0\\
0&0&0&0&0&-1\\
0&0&{a}^{-1}&0&0&0\\
0&0&0&a&0&0\\
1&0&0&0&0&0\\
0&-1&0&0&0&0
\end {array} \right].
$$

\section{Formulas for evaluations}
We now present the formulas which were used for the evaluations (on Maple) of the Nijenhuis tensor and curvature tensor of the associated metrics. Let $e_1,\ldots,e_{2n}$ be a basis of the Lie algebra $\mathfrak g$ and $C_{ij}^k$ a structure constant of the Lie algebra in this base:
\begin{equation}
[e_i,e_j]=\sum_{k=1}^{2n}C_{ij}^{k}e_k,  \label{strukt}
\end{equation}

\textbf{1. Nijenhuis tensor.}
Let $J^{k}_{i}$ be a matrix of a left-invariant almost complex structure $J$ in basis $\{e_i\}$,\ $Je_i=J^{k}_{i}e_k$. The Nijenhuis tensor is defined by formula (\ref{Nij1}). For the basis vectors we obtain: $N(e_i,e_j) = N_{ij}^k e_k$,
$$
N(e_i,e_j)= [Je_i,Je_j] -[e_i,e_j] -J[Je_i,e_j] -J[e_i,Je_j]=
$$
$$
=\left(J_i^l J_j^m C_{lm}^k  -J_i^l J_m^k C_{lj}^m   -J_j^l J_m^k C_{il}^m  -C_{ij}^k \right)\,e_k,
$$
\begin{equation}
N_{ij}^k =J_i^l J_j^m C_{lm}^k  -J_i^l J_m^k C_{lj}^m   -J_j^l J_m^k C_{il}^m  -C_{ij}^k.
\end{equation}

\textbf{2. Compatible condition.}
This is the condition that $\omega(JX,Y) + \omega(X, JY) =0$, $\forall \, X, Y\in \g $. For the basis vectors we have:
$\omega (J (e_i), e_j) + \omega (e_i, J (e_j)) =0$, \quad
$\omega (J^k_i e_k, e_j) + \omega (e_i, J^s_j e_s) =0$.
\begin{equation}
\omega _ {k, j} J^k_i + \omega _ {i, s} J^s_j=0.
\end{equation}

\textbf{3. Connection components.} These are the components $\Gamma_{ij}^{k}$ in the formula
$\nabla _ {e _ {i}} e_j =\Gamma_{ij}^{k}e_k.$
For left-invariant vector fields we have: $2g ({\nabla}_{X} Y, Z) =g ([X, Y], Z) +g ([Z, X], Y)-g ([Y, Z], X)$. For the basis vectors we have:
$$
2g (\nabla _ {e _ {i}} e_j, e_k) =g ([e_i, e_j], e_k) +g ([e_k, e_i], e_j) +g (e_i, [e_k, e_j]),
$$
$$
2g _ {lk} \Gamma _ {ij} ^ {l} =g _ {pk} C _ {ij} ^ {p} +g _ {pj} C _ {ki} ^ {p} +g _ {ip} C _ {kj} ^ {p},
$$
\begin{equation}
\Gamma_{ij}^{n}=\frac{1}{2}g^{kn}\left(g_{pk}C_{ij}^{p}+g_{pj}C_{ki}^{p} +g_{ip}C_{kj}^{p}\right).
\end{equation}

\textbf{4. Curvature tensor.}
The formula is: $R(X, Y)Z =\nabla_{X} \nabla_{Y} Z -\nabla_{Y}\nabla_{X}Z -\nabla_{[X,Y]}Z$.
For the basis vectors we have: $R (e_i, e_j) e_k=R _ {ijk} ^s e_s $,
$$
R(e_i,e_j)e_k=\nabla_{e_{i}}\nabla_{e_{j}}e_{k}-\nabla_{e_{j}}\nabla_{e_{i}}e_{k} -\nabla_{[e_{i},e_{j}]}e_{k}.
$$
Therefore:
\begin{equation}
R_{ijk}^{s}=\Gamma_{ip}^{s}\Gamma_{jk}^{p}-\Gamma_{jp}^{s}\Gamma_{ik}^{p} -C_{ij}^{p}\Gamma_{pk}^{s}.
\end{equation}

\end{document}